\documentclass[reqno,b5paper]{amsart}
\usepackage{amsmath}
\usepackage{amssymb}
\usepackage{amsthm}
\usepackage{enumerate}
\usepackage[mathscr]{eucal}
\usepackage{adjustbox}
\usepackage{float}
\usepackage{longtable}
\usepackage{lipsum}
\usepackage{amssymb}
\usepackage{mathtools}
\usepackage{physics}

\def\mod#1{{\ifmmode\text{\rm\ (mod~$#1$)}
\else\discretionary{}{}{\hbox{ }}\rm(mod~$#1$)\fi}}

\usepackage{float}
\restylefloat{table}
\newcommand{\pdiv}{\mid\!\mid}

\newtheorem{theorem}{Theorem}[section]
\newtheorem{lemma}[theorem]{Lemma}

\newtheorem{proposition}[theorem]{Proposition}
\newtheorem{conjecture}[theorem]{Conjecture}

\newtheorem{remark}[theorem] {Remark}

\numberwithin{equation}{section}

\def\mod#1{{\ifmmode\text{\rm\ (mod~$#1$)}
\else\discretionary{}{}{\hbox{ }}\rm(mod~$#1$)\fi}}

\usepackage{hyperref}

\begin{document}

\title{\uppercase {On the exceptional solutions of Je\'{s}manowicz'
conjecture}
}

\author{\uppercase{Amir Ghadermarzi}}

\address{
School of Mathematics, Statistics and Computer Science, College of Science, University of Tehran, Tehran, Iran
 \email a.ghadermarzi@ut.ac.ir }

\maketitle

\begin{abstract}
Let $(a,b,c)$ be a primitive Pythagorean triple. Set $a = m^2-n^2 $, $b=2mn$ , and $c=m^2+n^2$ with $m$ and $n$ positive coprime integers, $ m>n $ and $m \not \equiv n \pmod 2 $. A famous conjecture of Je\'{s}manowicz asserts that the only positive integer solution to the Diophantine equation $a^x+b^y=c^z$ is $(x,y,z)=(2,2,2).$ In this note, we will prove that for any $ n>0 $ there exists an explicit constant $c(n)$ such that if $ m> c(n) $, the above equation has no exceptional solution when all $x$,$y$ and $z$ are even. Our result  improves that of Fu and Yang \cite{zoj}. As an application, we will show that if $4 \pdiv m $ and $m > C(n) $, then Je\'{s}manowicz' conjecture holds.

\end{abstract}

\section{Introduction}
\label{intro}
Let $a$,$b$ and $c$ be coprime positive integers bigger than 1. In 1933, Mahler \cite{Mah} proved that the Diophantine equation
\begin{equation} \label{eq1}
a^x+b^y=c^z
\end{equation}
has only finitely many solutions. Mahler's proof depends on the $P$-adic analog of the Thue-Siegel method and is not effective in the sense that it gives no upper bound on the size of solutions or the number of solutions. Gel’fond, \cite{Gel} Using $p$-adic analysis of related expressions, obtained effective results. One can consider equation \eqref{eq1} as a special kind of $S$-unit equation and find an upper bound for the number of solutions by, for example, using the result of Beukers and Schlickewei \cite{Beu}. Baker's theory of linear form in logarithms can be effective in finding bounds on the number of solutions. By combining the Gel’fond–Baker method with an elementary approach, Hua and Le \cite{Hua} proved that if $\max \{a, b, c \} >5 \times 10^{27}$, then the equation $a^x+b^y=c^z$ has at most three positive integer solutions $(x, y, z)$. Moreover they showed that if $ \max \{a,b,c \} \geq 10^{62}$, then the equation \eqref{eq1} has at most two solutions in positive integers $x, y$ and $z$ \cite{Hua2}. It seems that the latter result holds with no restriction on values $a$,$b$ and $c$. Scott and Styer made even a stronger conjecture, i.e. the equation \eqref{eq1} has at most one solution except for some known cases, all listed in \cite [Conjecture 3.3] {Sco}.

One approach in determining solutions of the equation \eqref{eq1} is assuming the existence of one particular solution. To be more precise, assume $a^p+b^q=c^r$ with some conditions on $p$,$q$, and $r$, then disapprove the possibility of the existence of any other solution.There are much work, results, and conjectures by several authors in this direction (e.g., \cite{Ter}, \cite{Cao} , \cite{Cao2}...), which lead to the following conjecture \cite{LeConj}:
\begin{conjecture}
For given coprime integers $a, b, c > 1$, the Diophantine
equation \eqref{eq1} has at most one solution in integers $x, y, z > 1.$
\end{conjecture}
The most famous, old  example is when $a$, $b$, and $c$ are primitive Pythagorean triple, and hence $p=q=r=2$ is a solution. Sierpi\'{n}ski \cite{Sir} showed that the equation $3^x+4^y=5^z$ has a unique solution $(2,2,2)$ in positive integers. In the same year, Je\'{s}manowicz \cite{Jes} proved the same result for the Pythagorean triple $(a,b,c)= (5, 12, 13), (7,24,25)
, (9, 40, 41)$ and $(11; 60; 61)$. He further made the following conjecture:
\begin{conjecture} \label{jesmanzwitch}
Let $(a,b.c)$ be positive integers where $a^2+b^2=c^2$ and $GCD(a,b)=1$. Then the Diophantine equation $a^x+b^y=c^z$ has only the positive integer solution $(x,y,z)=(2,2,2)$.
\end{conjecture}
It is well known that if $(a,b,c)$ is a primitive Pythagorean triple, then there exist positive integers $m,n$ where $m>n $ , $GCD(m,n)=1$ and $ m \not \equiv n \pmod 2$ such that
$$ a=m^2-n^2 \quad b=2mn \quad c=m^2+n^2. $$
From now on, we consider the above expressions of $a$,$b$, and $c$. The Je\'{s}manowicz' conjecture is still open despite many efforts from different authors. Lu \cite{Lu} proved the conjecture \ref{jesmanzwitch} when $n=1$. Also De\'{m}janenco \cite{Dem} confirmed the conjecture for $m-n=1$. De\'{m}janenco's result was generalized by Miyazaki \cite{Miy1} and Miyazaki, Yuan, and Wu \cite{Miy2}. Using elementary methods, Le \cite{Le} showed that if $2 \pdiv mn$ and $C$ is a power of an odd prime $p$, then Je\'{s}manowicz' conjecture holds. On the other hand Using linear forms in logarithms method, Guo and Le \cite{Guo} proved that if $n = 3$, $2 \pdiv m$ and $m > 6000$, then 
Je\'{s}manowicz' conjecture is true. Many of the known results of conjecture \ref{jesmanzwitch} are concerned with the case $2 \pdiv mn$. A common theme in the proof of this case is to separate solutions with $y=1$ and $y>1$. Possible solutions with $y=1$  mostly are handled by elementary congruences or through considering the equation as a Pillai equation. While for y>1, one might try to show that for any possible solution of conjecture \ref{jesmanzwitch}, $x$, $y$, and $z$ all are even. This is achieved by assuming some conditions on the factorization of $a$,$b$ and $c$, elementary congruences, the quadratic reciprocity, and biquadratic character theory \cite{Le,Ter2,Miy3,qua1,qua2}. Parity of solutions is of interest because when $2 \pdiv mn$ and all $x$, $y$ and $z$ are even, it is easy to see that $x=y=z=2$ \cite[proof of lemma 2]{Guo} and  Je\'{s}manowicz' conjecture holds. If $ 4 \mid mn$, it is not straightforward to obtain the conclusion even if we assume the parity of solutions. For these reasons, in this note, we consider the solutions of the conjecture \ref{jesmanzwitch} where $x \equiv y \equiv z \equiv 0 \pmod 2 $. A solutions $(x,y,z)$ of Je\'{s}manowicz' conjecture is called exceptional if  $(x,y,z) \neq (2,2,2)$. From now on to save space we call a solution exceptional if $(x,y,z) \neq (2,2,2)$ and $x$, $y$ and $z$ are positive even integers. Note that this is not the usual definition of an exceptional solution in \cite{Sco} and \cite{zoj}. Let $2 \pdiv mn$, then we know that conjecture \ref{jesmanzwitch} holds when $y>1$ \cite{Ma},  $m+n$ has a prime factor $p$ with $ p \not \equiv 1 \pmod {16} $ \cite {Han} and when $ m > 30.8 n $ \cite{Soy}. Thus, from now on, we assume $ 4 \mid mn$. As mentioned before, assuming $ 4 \mid nm$ there are not as many results as the case $ 2 \pdiv mn$ assuming the parity of values $x$,$y$, and $z$. We would like to mention a result of Miyazaki \cite{Miym} that will be used later. To state Miyazaki's result, we need to recall some notation form \cite{Miym}. 
 
Let $n>1$. Define positive integers $\alpha$ and $\beta$ with $\beta \geq 2$ as follows: 
 \begin{equation} \label{alpha}
   \left\{\begin{array}{lll }
     m=2^{\alpha} i,  \quad \quad   &n=2^{\beta} j+e \text{} &\text {   if } m \text { is even, } \\ 
     m=2^{\beta} j+e, \quad &n=2^{\alpha} i \quad \quad &\text { if } m \text { is odd, } 
     \end{array}\right. 
   \end{equation}  
   where $e=\pm 1$ and $i, j$ are odd positive integers.
   
   \begin{lemma} \cite[Lemma 3.3]{Miym} \label{2alpha}
   Assume that $2\alpha=\beta+1$. If $(x,y,z)$ is a solution of conjecture \ref{jesmanzwitch}, $y$ is even and
$x \equiv z \pmod2$, then $(x,y,z)=(2,2,2)$.
\end{lemma}

Fu and Yang in \cite{zoj} proved the following result:
\begin{theorem} \cite [Theorem1.1]{zoj} \label{zoj}
If
$$ m > \max \{ 10^{127550},n^{5127},n^{(log(n))^2} \} $$
then conjecture \ref{jesmanzwitch} has no exceptional solutions $(x,y,z)$ with $x \equiv y \equiv z \equiv 0 \pmod 2 $.

\end{theorem}
 In this note, we will improve the above result by proving the next theorem
 \begin{theorem} \label{main thereom}
 If
$$ m > \max \{ 10^{109948},n^{(log(n)^{3/2}} \} $$
then conjecture \ref{jesmanzwitch} has no exceptional solutions $(x,y,z)$ with $x \equiv y \equiv z \equiv 0 \pmod 2 $.
 \end{theorem}
To compare this result with theorem \ref{zoj}, we will also show that:
\begin{theorem} \label{main2}
If
$$ m > \max \{ 10^{22933},n^{(log(n)^{2}} \} $$
then conjecture \ref{jesmanzwitch} has no exceptional solutions $(x,y,z)$ with $x \equiv y \equiv z \equiv 0 \pmod 2 $.
\end{theorem}
To prove theorems \ref{main thereom} and \ref{main2}, we follow the ideas of \cite{Miyasl}. For any exceptional solution $(x,y,z)$ of  Je\'{s}manowicz' conjecture, we consider the quantity $\Delta = z-x$. We will show that $\Delta$ is a positive integer. Using linear forms in two logarithms, we will find an upper bound for $\Delta$ in the order of $ (\log (\log(m))^2 $. Comparing this upper bound with trivial lower bound for $\Delta$ leads to the proof of theorems \ref{main thereom} and \ref{main2}.\\ 
In the second part of this note, we will show that if $ 4 \pdiv m$ and $ y> 1$, then $x \equiv y \equiv z \equiv 0 \pmod 2 $. Therefore, for any pair $(n,m)$  that $ 4  \pdiv m $ and satify the conditions of theorem \ref{main thereom} or \ref{main2}, Je\'{s}manowicz' conjecture holds. Note that most arguments in section \ref{section4} are valid for $ 4 \mid m$.

As the final remark of this section, since Je\'{s}manowicz' conjecture has been proved to be true for $n=2 $ \cite{Ter2} and $ n=3 $ \cite{Miyasl}, we assume $ n \geq 4$. 
\section{bounds for $x$,$y$ and $z$} \label{sec:1}
We assume $(x,y,z)$ is an exceptional solution of  Je\'{s}manowicz' conjecture. As we mention before, we mean $(x,y,z) \neq (2,2,2) $ and $x$, $y$ and $z$ are positive integers and all even. For an exceptional solution $(x,y,z)$, we define $X=\frac{x}{2}, Y=\frac{y}{2}$ and $Z=\frac{z}{2}$. It is an easy observation that if $x$,$y$ and $z$ are even and $(x,y,z) \neq (2,2,2)$, then $x,y>2$ and therefore $z>2$. (see proof of \cite[Lemma 2.2]{Ma}).
In this section, we will prove the following lemma, which is essential in the proof of theorems \ref{main thereom} and \ref{main2}.
\begin{lemma} \label{karan}
Let $(x,y,z)$ be an exceptional solution of  Je\'{s}manowicz' conjecture, then
$$ Y < \frac{\log (n)}{\log 3} \quad \text {or} \quad Y < \frac{\log 2(m-1)}{(\alpha+1) \log 2} $$
where $\alpha$ is as defined in \ref{alpha} , moreover if $m >1.22 n$ then $x<y<z $.
\end{lemma}
\subsection{Congruence conditions of soloutions} \label{sec:1.2}
Based on our observation of the values of $x$,$y$, and $z$ of an exceptional solution, we can consider the exceptional solutions of Je\'{s}manowicz' conjecture as solutions of generalized Fermat equations of the shape:
$$ A^{p} +B^{q}= C^{r} \quad \frac{1}{p}+\frac{1}{q}+\frac{1}{r} <1 $$
Where $A$, $B$, and $C$ are coprime, and $X_i$s are positive integers. The results on the solutions of such type of the Diophantine equations can be very helpful in determining congruence conditions on exceptional solutions. We quote some of these results.
\begin{proposition} \cite [Theorem1]{Ben1}
There are no solutions in coprime integers $A, B, C$ to the equation $A^4+B^2=C^N$ with  $N \geq 4$
\end{proposition}
From this proposition, we can conclude that $X=\frac{x}{2}$ and $Y=\frac{y}{2}$ are both odd.

 \begin{proposition}  \cite [Theorem3]{Cao2}
If $N \in \mathbb{N}$ with $N >1$, then the Diophantine equation
$$ A^{2N}+B^2=C^4 , \quad  A,B,C \in \mathbb{Z},   (A,B)=1, 2 \mid A $$
has no integral solutions.
 \end{proposition}
From this proposition, we can conclude that $Z=\frac{z}{2}$ is odd.
\begin{proposition} \cite[Theorem1]{Ben2} \label{z,15}
Suppose that $(p,q,r)$ are positive integers with $ \frac{1}{p}+\frac{1}{q}+\frac{1}{r} <1 $ and 
$(p,q,r) \in \{ (2,n,6),(2,2n,10) \}$ for some integer $n$. Then the equation $A^x+B^y=C^z$ has no solutions in coprime nonsero integers $A$,$B$ and $C$.
\end{proposition}
From this proposition, we conclude that $gcd(z,15)=1$.
\subsection{Proof of lemma \ref{karan}} \label{sec:1.3}
Let $(x,y,z)$ be an exceptional solution. Then $a^X,b^Y$, and $c^Z$ form a primitive Pythagorean triple. So there exist positive integers $k$, and $l$ with $k > l$, $ k \not \equiv l \pmod 2 $ and $(k,l)=1$ such that:
$$ a^X= k^2-l^2, \quad b^Y= 2kl \quad c^Z=k^2+l^2. $$
\begin{lemma}
$z<2y$ and $z<2x$
\end{lemma}
\begin{proof}
Since $a^X,b^Y$ and $c^Z$ form a primitive Pythagorean triple. $c^Z < b^{2Y}$ and $c^Z < a^{2X}$, but $ c>a $ and $c >b$. Therefore $Z< 2Y$ and $Z< 2X.$ \qed
\end{proof}
We can factorize $c^z=k^2+l^2 $ in the ring $ \mathbb{Z}[i] $ then we have 
$$ c^Z = \left (k+li \right ) \left( k-li \right) $$
Since $c$ is odd and $(k,l)=1$, the factors on the right-hand side are relatively prime in $ \mathbb{Z}[i] $. Since $Z$ is odd, there exist nonzero integers $a_1$ and $b_1$ such that 
$$ k+li = \left ( a_1+b_1 i \right)^Z $$
$a_1^2+b_1^2=m^2+n^2=c$, $a_1 \not \equiv b_1 \pmod 2 $, $k= \Re {\left ( a_1+b_1 i \right)^Z }$ and $l= \Im {\left ( a_1+b_1 i \right)^Z }$.
it means that:
\begin{equation} \label{k}
k= a_{1}\left(a_{1}^{z-1}-\left(\begin{array}{c} Z \\ Z-2 \end{array}\right) a_{1}^{Z-3} b_{1}^{2}+\cdots \mp \left(\begin{array}{c} Z \\ 3 \end{array}\right) a_{1}^{2} b_{1}^{Z-3} \pm Z b_{1}^{Z-1}\right), 
\end{equation}
 \begin{equation} \label{l}
  l=b_{1}\left(Z a_{1}^{Z-1}-\left(\begin{array}{c} Z \\ Z-3 \end{array}\right) a_{1}^{Z-3} b_{1}^{2}+\cdots \mp \left(\begin{array}{l} Z \\ 2 \end{array}\right) a_{1}^{2} b_{1}^{Z-3} \pm b_{1}^{Z-1}\right).
 \end{equation}
  From these equations, it is clear that $a_1 \mid k$ and $ b_1 \mid l$. Since $ a_1 \not \equiv b_1 \pmod 2 $ and $Z$ is odd, $k/a_1$ and $l/b_1$ are odd. Assume $b_1$ is even (The argument is exactly the same when $a_1$ is even.), then $l$ is even, $k$ is odd and by \ref{l} $ Ord_2 \left ( b_1\right )=Ord_2 \left( l \right) $. Therefore, 
$$ \left( \alpha+1 \right) Y = Ord_2 (b^Y)= Ord_2 (2kl)= Ord_2(2b_1) $$
Since $Ord_2(b_1) > Ord_2(mn)$, we have $b_1 \neq m$ and $b_1 \neq n$. 
\begin{remark}
$c=m^2+n^2=a_1^2+b_1^2$, and $mn \neq a_1b_1$. Therefore, $c$ can not be a prime power.
\end{remark}
We consider two cases: 
\begin{itemize}
\item[•] $b_1<m$:
Since $ Ord_2(2b_1) = \left( \alpha+1 \right) Y$ we have $Y < \frac{\log (2(m-1))}{(\alpha+1) \log(2)}$
\item[•] $b_1>m$: Assume $ b_1 >m$ then $ a_1<n $. we will show that $a_1 \geq 3^Y $.
\end{itemize}
\begin{lemma}
$a_1 \neq 1.$
\end{lemma} 
\begin{proof}
Assume $a_1=1$ and $m$ is even (the proof for $n$ even is exactly the same). Then $ b_1^2= m^2 +(n^2-1) $. Considering the 2-adic valuation of $b_1$ and $m$, we can conclude that $ Ord_2 (m^2) = Ord_2 (n^2-1)$. On the other hand, from \eqref{alpha},
we have $ Ord_2( m^2 )= 2 \alpha $ and $ Ord_2 (n^2-1) = \beta+1$. From lemma \ref{2alpha} if $\beta+1 = 2 \alpha$ , we don't have any exceptional solution.  \qed
\end{proof}
Let $p$ be a prime divisor of $a_1$. Remember that  $a_1$ is odd, bigger than 1 ,and $a_1 \mid k$. Consider the equation \eqref{k}. It is proved in \cite[p.~311-312]{Miy2} that
$$ \operatorname{ord}_{p}\left(\left(\begin{array}{c} Z \\ Z-i \end{array}\right) a_1^{i-1} \right)>\operatorname{ord}_{p}(Z) $$ 
therefore,
  $$ \operatorname{ord}_{p}(k) = \operatorname{ord}_{p}(a_1)+\operatorname{ord}_{p}(Z). $$  
  if $ p \in \{ 3,5 \} $ from proposition \ref{z,15}, we have $\operatorname{ord}_{p}(a_1) =\operatorname{ord}_{p}(k) \geq Y$, so $3^Y <n$. Let $ p >5$ , $\operatorname{ord}_{p}(Z) \leq \frac{\log (Z)}{\log p} < \frac{\log 2Y}{\log p}$. Thus, if $p$ divides $a_1$ and $p>5$, we obtain $p^{Y-\log(2Y)/ \log(p)} <n$. However, for every $p>5$ and $Y>1$, $3^Y<p^{Y-\log(2Y)/ \log(p)}$ . This completes the proof of the first part of lemma \ref{karan}.
To complete the proof of lemma \ref{karan}, we need to compare the values of $x$, $y$ and $z$ of an exceptional solution. Note that these relations have already been proved in \cite{Miyasl} but for the sake of completeness and also have the more accurate constants in our case, we will state the following lemmas:
\begin{lemma}
For any exceptional solution $(x,y,z)$, we have $ x \neq z $, hence, $ \left \lvert x-z \right \rvert \geq 4 $.
\end{lemma}
\begin{proof}
Assume $x=z$, then $c^x-a^x= b^y$. Therefore, $\left( c^2 \right) ^X - \left (a^2 \right)^X =b^y$. Since $X$ is odd and, $a^2 \equiv c^2 \equiv 1 \pmod 4 $, we have $\operatorname{ord}_2 (b^y)= \operatorname{ord}_2 \left ( \left( c^2 \right) ^X - \left (a^2 \right)^X \right) = \operatorname{ord}_2 \left( c^2-a^2 \right) = \operatorname{ord}_2 (b^2)$ and, since $b$ is even, it means $y=2$. This contradiction proves the lemma.\qed
\end{proof} 

\begin{lemma}
Let $m > 1.22 n$ and $(x,y,z)$ be an exceptional solution, then $ x <z$.
\end{lemma}
\begin{proof}
Assume $x>z$. Since $a^x< c^z$ and $a^X< b^{2Y}$, we have 
\[ 
\begin{array} {rcl}
 X\log (a) < Z log(c) &\Rightarrow X \log (a) &\leq (X-2) \log(c) \\ 
& \Rightarrow 2 \log (c) &\leq X \log \left (\frac{c}{a} \right) \\
&  \Rightarrow 2 \log (c) &<2Y \frac{\log(b)}{\log(a)} \log \left( \frac{c}{a} \right) \\
& \Rightarrow \log(c) &< \frac{\log(2 \sqrt{c-1})}{(\alpha+1) \log(2)} \frac{ \log b}{\log a} \log \left( \frac{c}{a} \right)
\end{array}
\]
Note that we assumed $ \alpha \geq 2 $, $n \geq 4 $, $ 2 \alpha \neq \beta+1$, $c$ is not a prime power, and from De\'{m}janenco's result \cite{Dem} $m-n \geq 3$. A small search shows that the smallest value of $c$ with these conditions is 185, which occurs for $(n,m)=(4,13)$ and $(n,m)=(8,11)$. For $c>185$ and $ \alpha \geq 2 $, the last inequality does not hold when $\frac{m}{n} > 1.217$. \qed
\end{proof}
\begin{remark} \label{nokte}
Since $x<z$, we have $a^x<c^{z-4}<\frac{c^z}{2}$. Therefore, $a^x<b^y$.
\end{remark}
\begin{lemma} 
$z<y$
\end{lemma}
\begin{proof}
Assume, to the contrary that, $ z \leq y$ then since $z > x $, we have
$$ c^z = \left( c^2\right)^Z= (a^2+b^2)^Z > a^{2Z}+b^{2Z} > a^x+b^y =c^z \qed $$
\end{proof}

We finish this section by bounding $z$ in terms of $y$. This bound will be helpful in proving theorems \ref{main thereom} and \ref{main2}. From remark \ref{nokte}, we obtain $c^z<2 b^y$, so
\begin{equation} \label{z,y}
 z \log(c) < \log(2)+y \log (b) \Rightarrow z<\frac{ \log(2)+y \log b}{\log c} 
\end{equation}

\section {Linear form in two logarithms}
For any algebraic number $\gamma$ of degree $d$, we define the absolute logarithmic height of $\gamma$ by the following formula:
$$\mathrm{h}(\gamma)=\frac{1}{d}\left(\log |a|+\sum_{\sigma} \log \max \left(1,\left|\sigma(\gamma) \right|\right)\right),$$
where $a$ is the leading coefficient of the minimal polynomial of $\gamma$ over $\mathbb{Z}$ and $\sigma$s are the embeddings of $\mathbb{Q}(\gamma)$ into $\mathbb{C}$. Consider the linear form in logarithms: 
$$ \Lambda = b_2 \log \alpha_2 - b_1 \log \alpha_1, $$
where $b_i$s are positive integers, $\alpha_i$s are algebraic numbers, and $\log (\alpha_i)$ is any determination of logarithms. Assume $\left \lvert \alpha_1 \right \rvert $, $\left \lvert \alpha_2 \right \rvert \geq 1$. Define 
\[ D=\left[\mathbb{Q}\left(\alpha_{1}, \alpha_{2}\right): \mathbb{Q}\right] / \left[\mathbb{R}\left(\alpha_{1},\alpha_{2}\right): \mathbb{R} \right] \]
The following  result of Laurent \cite{Lau} gives a lower bound for $\log \abs{\Lambda}$. Note that  $ \alpha_1$ and $\alpha_2$ are not necessarily multiplicatively independent. 
 \begin{theorem} \cite [Theorem1 ] {Lau} \label{Lau}
Let $K$ be an integer $\geq 2,$ and $L, R_{1}, R_{2}, S_{1}, S_{2}$ be positive
integers. Let $\varrho$ and $\mu$ be real numbers with $\varrho>1$ and $1 / 3 \leq \mu \leq 1 .$ Put

\[ R=R_{1}+R_{2}-1, \quad S=S_{1}+S_{2}-1, \quad N=K L, \quad g=\frac{1}{4} - \frac{N}{12RS}, \]
$$\sigma=\frac{1+2 \mu-\mu^{2}}{2}, \quad b=\frac{(R-1) b_{2}+(S-1) b_{1}}{2}\left(\prod_{k=1}^{K-1} k !\right)^{-2 / \left(K^{2}-K\right).}$$
Let $a_{1}, a_{2}$ be positive real numbers such that 
\[ a_{i} \geq \varrho\left|\log \alpha_{i}\right|-\log \left|\alpha_{i}\right|+2 D \mathrm{h}\left(\alpha_{i}\right) \]
for $i=1,2 .$ Suppose that 
\[ \operatorname{Card}\left\{\alpha_{1}^{r} \alpha_{2}^{s} ; 0 \leq r<R_{1}, 0 \leq s<S_{1}\right\} \geq L \] \[\operatorname{Card}\left\{r b_{2}+s b_{1} ; 0 \leq r<R_{2}, 0 \leq s<S_{2}\right\}>(K-1) L \]
and
\begin{equation*}
\begin{aligned}
 K(\sigma L-1) \log \varrho- &(D+1) \log N \\
&-D(K-1) \log b-g L\left(R a_{1}+S a_{2}\right)>\varepsilon(N),
\end{aligned}
\end{equation*}
where
\[ \varepsilon(N)=2 \log \left(N ! N^{-N+1}\left(e^{N}+(e-1)^{N}\right)\right) / N. \]
Then
\[ \left|\Lambda^{\prime}\right|>\varrho^{-\mu K L} \quad \text{with} \quad \Lambda^{\prime}=\Lambda \max \left\{\frac{L S e^{L S|A| /\left(2 b_{2}\right)}}{2 b_{2}}, \frac{L R e^{L R|\Lambda| /\left(2 b_{1}\right)}}{2 b_{1}}\right\}.\]
\end{theorem}
For our purpose, we consider the case where $\alpha_1=-1$ and $\alpha_2$ is an algebraic number of degree 2, with $\left \lvert \alpha_2 \right \rvert =1$, but not a root of unity.  
 \begin{lemma} \label{cor}
 Let $\theta$ be an algebraic number of degree 2, with $\left \lvert \theta \right \rvert $=1, but not a root of unity and $\log \theta $ is any determination of its logarithm. Consider the linear form 
$$ \lambda = b_1 i \pi - b_2 \log \theta $$
where $b_1$ and $b_2$ are positive integers. Let $\varrho=\exp (3.1)$. Define the real numbers $a_1$, $a_2$ and $b'$ as 
$$ a_1= \varrho \pi \quad a_2 \geq \varrho \left \lvert \log \theta \right \rvert + 2 h(\theta) \quad b'=b_1/a_2+b_2/a_1 $$
Assume $a_2 \geq 1000+ a_1$ and $b'>0.056$, then 
$$ \log \Lambda > -3.741 \left( \log b'+ 6.87 \right)^2 a_2 -\frac{31L}{15} - \log(L) - \log \left(2+ 0.222 L a_2 \right) $$
where 
$$ L = \left[ \frac{45}{62} \left( \log b' +5.49 \right) \right] +1. $$
 \end{lemma}
\begin{proof}
 This is a corollary of theorem \ref{Lau}, with the suitable choice of parameters. 

 Put
 $$R_{1}=2, \quad S_{1}=[(L+1) / 2],\quad K=1+\left[ \kappa La_1A_2\right], $$
 and 
 \[  R_{2}=1+\left[\sqrt{(K-1) L a_{2} / a_{1}}\right], \quad S_{2}=1+\left[\sqrt{(K-1) L a_{1} / a_{2}}\right]. \]
We want to remark that these parameters are employed in \cite {App}, so all the bounds in the appendix of \cite{App} involving these parameters are valid here. For our case, we specialize the parameters $L$ and $\varrho$ as stated in the lemma and the values
$$ \mu = \frac{2}{3},\quad \kappa= 0.04927 \quad \sigma=\frac{17}{18}$$ 
Note that, from the conditions of the lemma, we obtain $L \geq 3$ and so $ K \geq 11030$ and $N>33090$. We have the following bounds, These bounds are obtained by applying our special parameter values to the bounds proved in \cite [page 33]{App}
$$ gL (Ra_1+Sa_2) < K \left( \frac{31}{20} L + 0.0612 \right) $$
$$ \log (b) \leq \log b' +2.3264 .$$ 
Therefore,
\begin{equation*}
 \begin{aligned}
 K(\sigma L-1) \log \varrho
&-D(K-1) \log b-g L\left(R a_{1}+S a_{2}\right) \\
& > K \left[ \left(\sigma \log \varrho -\frac{31}{20} \right) L- \log b' - 5.4876\right] + \log b. \\
\end{aligned}
\end{equation*}
Also,
$$ \varepsilon(N) \leq \frac{2}{N}\left(\frac{3}{2} \log N+\frac{1}{2} \log (2 \pi)+\frac{1}{12 N}+\log \left(1+\left(\frac{e-1}{e}\right)^{N}\right)\right). $$
The right-hand side is a decreasing function for $ N > e $, hence $\varepsilon(N) < 0.0011.$ With this set of parameters, the conditions of theorem \ref{Lau} are satisfied. We conclude that:
 \begin{equation*}
 \begin{split}
  \log \left( \left|\Lambda^{\prime}\right| \right)>& - \mu KL\varrho \\
                   &-3.741 \left( \log b'+ 6.87 \right)^2  a_2 -\frac{31L}{15}. 
 \end{split}
 \end{equation*}
On the other hand:
\begin{equation*}
\begin{split}
 \max \left\{\frac{L S e^{L S|A| /\left(2 b_{2}\right)}}{2 b_{2}}, \frac{L R e^{L R|\Lambda| /\left(2 b_{1}\right)}}{2 b_{1}}\right\} < & LR \\
 <& L(2+ \sqrt{\kappa} a_2 L). 
\end{split}
\end{equation*}
This completes the proof of lemma \ref{cor}. \qed
\end{proof}

\section{Proof of theorems \ref{main thereom} and \ref{main2}}
\label{section4}
We will prove theorems \ref{main thereom} and \ref{main2} by comparing the lower bound and upper bound of $\Delta$. First, note that if $ m > n^{\log (n)^{3/2}}$ since $ n>3$, we have $m> 1.22 n$. So from lemma \ref{karan}, we have 
$x<y<z$ and $a^x < b^y$. 
\begin{lemma} \label{karan paein}
$ \frac{\log m}{\log n } < x-z$
\end{lemma}
\begin{proof}
$a^2=b^2=c^2$ and $x$ is even, so $a^x \equiv c^x \pmod{ b^2} $ . $a^x+b^y= c^z$ and $y > 2$ so $c^z \equiv a^x \pmod {b^2}$. \\ Therefore, $ (m^2+n^2)^z \equiv (m^2+n^2)^x \pmod {(2mn)^2} \Rightarrow (n^2)^x \equiv (n^2)^z \pmod {m^2} $
Since $ n < m$, the lemma is proved. \qed
\end{proof}
Assume $U$,$V$, and $W$ are coprime positive integers with $W$ is odd and $U$ is even. Let 
$$U^{2}+V^{2}=W^{T},$$
Then, by factorizing the above equation in $ \mathbb{Z}[i]$, we conclude that there exist coprime positive integers $u$ and $v,$ and $\lambda_{1}, \lambda_{2} \in\{\pm i, \pm -1\}$ such that $$ U+V i=\lambda_{1} \varepsilon^{T}, \quad \varepsilon=u+v \lambda_{2} i, \quad W=u^{2}+v^{2} $$ Moreover, if $\varepsilon=|\varepsilon| e^{\theta i / 2}$ then
\begin{equation} \label{sin}
U=W^{T / 2}|\cos (T \theta / 2)|, \quad V=W^{T / 2}|\sin (T \theta / 2)| 
\end{equation} 
We closely follow an argument of Cipu and Mignotte \cite{Mig} to find a lower bound for $ \min \{U,V\} $.  Let $k$ be the closest integer to $T \theta / \pi$. In other words $\abs{T \theta-k \pi}=   \min _{k^{\prime} \in \mathbb{Z}}\abs{T \theta-k^{\prime} \pi}$. Then from the equation \eqref{sin}, we have 
 \[ \min \{U, V\} \geq \frac{W^{T / 2}}{\pi} \abs{T \theta-k \pi}. \] 
On the other hand, put:
  $$ \psi:=\varepsilon / \bar{\varepsilon}=\frac{u^{2}-v^{2}+2 u v i}{u^{2}+v^{2}} e^{\theta i}, $$
 $\psi$ satisfies the equation: 
 $$ \left(u^{2}+v^{2}\right) (\psi)^{2}-2\left(u^{2}-v^{2}\right) (\psi)+\left(u^{2}+v^{2}\right)=0. $$
 So $\left\lvert \psi \right\rvert =1 $, $ \psi $ is an algebraic number of degree 2, it is not a root of unity and the absolute logarithmic height of $\psi$ is $$ h(\psi)=\frac{1}{2} \log \left(u^{2}+v^{2}\right)=\frac{1}{2} \log W. $$
  Define the following linear form in logarithms 
  $$ \Lambda=T \log \psi-k \log (-1) $$
  with the principal determination of the logarithm. Since $\psi$ is not the root of unity,$ \Lambda \neq 0$. 
  $\Lambda= \left( T \theta - k \pi \right) i$ , and 
    $$ \min \left\{U, V\right\} \geq \pi^{-1}W^{T / 2}|\Lambda|, $$
    where $k$ is a positive integer less than T. 
    
 \begin{lemma} {\label{karan bala}}
 Let $x,y,z$ be an exceptional solution of Je\'{s}manowic\'{z} conjecture. Then $ z-x < 2 \frac{-\log (\abs{\Lambda} )+ \log \pi} {\log (c)} $, where $ \lambda = k i \pi - z\log \theta $, $\theta$ is an algebraic number of degree 2,  $ h( \theta) = \log (c) /2$ and $k$  is a positive number less than $z$. 
 \end{lemma}
 \begin{proof}
It immediately follows from the argument prior to the lemma. $\left( a^X \right)^2+ \left(b^Y \right)^2= c^z$. Therefore,
$a^X > c^Z / \pi \abs{\Lambda} $. By tacking logarithms from both sides, the lemma is proved. \qed
 \end{proof}   
Now we are ready to prove theorems \ref{main thereom} and \ref{main2}. For any $m$ with $\log m > 1000$ the conditions of lemma \ref{cor} are satisfied. From lemmas \ref{cor}, \ref{karan paein} and \ref{karan bala}, we have 
$$ \frac{\log m }{ \log n } \log (c) < 3.741 \left( \log b'+ 6.87 \right)^2 a_2 +\frac{31L}{15} + \log(L) +\log \left(2+ 0.222 L a_2 \right)+ \log (\pi), $$
where $ a_2 =\log(c)+ \exp{3.1}\pi $ , $b' \leq z \left(\frac{1}{69.73}+ \frac{1}{a_2}\right)$ and
$L = \left[ \frac{45}{62} \left( \log b' +5.49 \right) \right] +1$
From the equation \ref{z,y} under the assumption $ m >n^{\log n}$ and $ \log (m) >32 $, we have $z < \frac{3}{5} y$, and from the equation \eqref{karan} 
$$ y < 2\frac{\log (n)}{\log 3} \quad \text {or} \quad y < 2\frac{\log 2(m-1)}{(\alpha+1) \log 2}. $$
In our case, the former bound gives a much sharper result so by considering the later bound and above estimates, we have 
$$ \log(b') < \log ( \log (2m) )-4.731. $$ 
 To summerize, under the assumption $ \log (m ) >1000$. if $(x,y,z)$ is an exceptional solution of Je\'{s}manowicz conjecture, then 
 \begin{equation*}
 \begin{split}
\frac{\log (m)}{ \log(n) } < 7.482(\log(2m)+2.139)^2 & \left(1+\frac{70} {\log(2m)}\right)+\frac{62}{15 \log c} L^{\prime}\\ &+\frac{2}{\log c} \left( \log 6.29 L^{\prime}  +0.7 L^{\prime^{2}} (\log(m)+70)) \right), 
\end{split}
 \end{equation*}
 where  $ L^{\prime}= \frac{45}{62} \log (\log 2m) +1.56. $ 
 
 Assume $ m > n^{\log (n)^{3/2}} $. Then if Je\'{s}manowicz conjecture has an exceptional solution, we obtain
 \begin{equation*}
 \begin{split}
\log (m)^{3/5}< 7.482(\log(2m)+2.139)^2 & \left(1+\frac{70} {\log(2m)}\right)+\frac{31}{15 \log m} L^{\prime}\\ &+\frac{\left( \log 6.29 L^{\prime}  +0.7 L^{\prime^{2}} (\log(m)+70)) \right)}{\log m} , 
\end{split}
 \end{equation*}
 The inequality does not hold for $m>10^{109948}$ this proves theorem \ref{main thereom}. 
 
 Now assume  $ m > n^{\log (n)^{2}} $ then if Je\'{s}manowicz conjecture has an exceptional solution, we obtain
 \begin{equation*}
 \begin{split}
\log (m)^{2/3}< 7.482(\log(2m)+2.139)^2 & \left(1+\frac{70} {\log(2m)}\right)+\frac{31}{15 \log m} L^{\prime}\\ &+\frac{\left( \log 6.29 L^{\prime}  +0.7 L^{\prime^{2}} (\log(m)+70)) \right)}{\log m} , 
\end{split}
 \end{equation*}
 The inequality does not hold for $m>10^{22933}$ this proves the theorem \ref{main2}. 
 \qed
 
 \section{ $4 \pdiv m$} \label{section akhar}
 With the same notations as previous sections, assume $a=m^2-n^2, b=2mn$ and $c=m^2+n^2$  be primitive Pythagorean triple. In this section, we will prove that if $4 \pdiv m $ and the conditions of theorems \ref{main thereom} or \ref{main2} satisfy, then Je\'{s}manowicz' conjecture holds. For this purpose, we will show that if $4 \pdiv m$ and 
\begin{equation} \label{bakhshakhar}
\left( m^2-n^2 \right)^x + \left( 2mn \right) ^y = \left( m^2+n^2 \right)^z, 
\end{equation}
then either $y=1$ or all $x,y$ , and $z$ are even. 
\begin{lemma}
Assume $(x,y,z)$ is a solution of \eqref{bakhshakhar}. If $ m$ is even, then $x$ is even.
\end{lemma}
Considering the equation \eqref{bakhshakhar} modulo 4, we have $ (-n^2)^x \equiv (n^2)^z \pmod 4$. Since $n$ is odd, $x$ is even.
\begin{lemma} \label{quadratic}
Assume $(x,y,z)$ is a solution of \eqref{bakhshakhar}, then
\begin{enumerate}
\item If $ m+ n \equiv 5 \pmod 8$, then $y \equiv z \pmod 2$
\item If $ m+n \equiv 7 \pmod 8$, then $y$ is even.
\item If $m+n \equiv 3 \pmod 8$, then $z$ is even. 
\item If $m-n \equiv 5 \pmod 8$ then $y \equiv z \pmod 2$
\end{enumerate}
\end{lemma}
\begin{proof}
Denote the Jacobi symbol by $ \left ( \frac{*}{*} \right) $. Consider the equation \eqref{bakhshakhar} modulo $m+n$. Then 
$ \left( \frac{-2n^2}{m+n} \right)^y= \left( \frac{-2n^2}{m+n} \right)^z \Rightarrow \left ( \frac{-2}{m+n} \right)^y= \left( \frac{2}{m+n}\right)^z.$ Similarly tacking the equation \eqref{bakhshakhar} modulo $m-n$, we have $ \left( \frac{2n^2}{m-n} \right)^y= \left( \frac{2n^2}{m+n} \right)^z \Rightarrow \left ( \frac{2}{m-n} \right)^y= \left( \frac{2}{m-n}\right)^z.$ The Lemma follows from quadratic reciprocity theorem. \qed
\end{proof}

\subsection{$4 \pdiv m$ and $ n \equiv 1 \pmod 8 $}
From part 1 of lemma \ref{quadratic}, $ y \equiv z \pmod 2$. Let $\left( \frac{*}{*}\right)_4$ denote the biquadratic residue symbol ( see \cite{Ire}, \cite{qua2}). Since $n \equiv 1 \pmod 4$ and $ m \equiv 0 \pmod 4$, $n-mi$ is an odd primary number.  From properties of the quartic character residue, we have 
$$ \left( \frac{i}{n-mi} \right)_4=1 \quad \quad  \left( \frac{-1}{n-mi} \right)_4=1 \quad  \quad \left( \frac{2}{n-mi} \right)_4=-1 $$
Let $m_1=m/4$ be an odd number. Consider the equation \eqref{bakhshakhar} modulo $n-mi$. Then 
 $$ \left( \frac{2n^2}{n-mi} \right)_4^x = \left( \frac{- (2m^2i)^y}{n-mi} \right)_4 $$
 From properties of biquadratic residue and biquadratic reciprocity, we have following equations:
 $$\left( \frac{2n^2}{n-mi} \right)_4= \left(\frac{2}{n-mi} \right)_4 \left( \left( \frac{n}{n-mi} \right)_4\right)^2 =-  \left( \left( \frac{n-mi}{n} \right)_4\right)^2 = -  \left( \left( \frac{mi}{n} \right)_4\right)^2 = -1   $$
\begin{equation*}
\begin{split} 
 \left( \frac{ (2m^2i)}{n-mi} \right)_4 = \left(\frac{2}{n-mi} \right)_4 \left( \frac{16m_1^2i}{n-mi} \right)_4=& - \left( \frac{m_1^2}{n-mi} \right)_4=- \left( \frac{n-mi}{m_1^2} \right)_4 \\
 = & - \left(\left( \frac{n-mi}{m_1} \right)_4\right)^2=- \left(\left( \frac{n}{m_1} \right)_4\right)^2=-1  
 \end{split}
 \end{equation*}
 Therefore, $(-1)^x=(-1)^y$ and if $x,y,z$ is a solution for \eqref{bakhshakhar}, where $4 \pdiv m$ and $n \equiv 1 \pmod8 $, then $x \equiv y \equiv z \equiv 0 \pmod 2. $

\subsection{$4 \mid m$ and $ n \equiv 3 \pmod 8 $}
From part 2 of lemma \ref{quadratic}, $ y $ is even. Since $ 4 \mid m $ and $n \equiv 3 \pmod 8 $, $ m^2 \equiv 0 \pmod {16} $ and $n^2 \equiv 9 \pmod {16} $. Consider the equation \eqref{bakhshakhar} modulo 16, we obtain $7^x \equiv 9^z \pmod {16}$. Therefore, $z$ is also even, and all $x$,$y$, and $z$ are even.

\subsection{ $4 \mid m$ and $ n \equiv 5 \pmod 8 $} 
 By \cite [lemma 9.5]{Miyasl} the equation \eqref{bakhshakhar} has no solution with $y=1$ when $ \frac{m}{n} >56$. Note that the condition $ \frac{m}{n} >56$ is much stronger than the conditions in theorems \ref{main thereom} and \ref{main2}, so we assume $ y > 1$. Again considering equation \eqref{bakhshakhar} modulo 16 with the assumption $y> 1$, we obtain $7^x \equiv 9^z \pmod{16}$, hence $z$ is even. We recall an intresting result of Miyazaki \cite{Miyf}.
\begin{proposition}
 Let $x, y, z$ be a solution of equation \eqref{bakhshakhar}. Assume that $x$ and $z$ are even. Then $X=x / 2$ and $Z=z/ 2$ are odd.
\end{proposition}
Since $x$ and $z$ are even we have :
$$ (2mn)^y= DE, $$
where
$$ D=\left(m^{2}+n^{2}\right)^{Z}+\left(m^{2}-n^{2}\right)^{X}, \quad E=\left(m^{2}+n^{2}\right)^{Z}-\left(m^{2}-n^{2}\right)^{X}. $$
Since $gcd(m,n)=1$ , $GCD(E,D)=2$ and $E \equiv 2 \pmod 4 $, so we obtain:
\begin{equation} \label{i}
\left(m^{2}+n^{2}\right)^{Z}-\left(m^{2}-n^{2}\right)^{X} = 2 m_1^y n_1^y,
\end{equation}
\begin{equation} \label{ii}
\left(m^{2}+n^{2}\right)^{Z}+\left(m^{2}-n^{2}\right)^{X} = 2^{(\alpha+1)y-1} m_2^y n_2^y,
\end{equation}
and 
\begin{equation} \label{iii}
\left(m^{2}+n^{2}\right)^{Z} = 2^{(\alpha+1)y-2} m_2^y n_2^y+ m_1^y n_1^y,
\end{equation}
where $\alpha \geq 2 $ is as defined in \ref{2alpha}, $m_1$ and $m_2 $ are odd and relatively prime and $m=2^{\alpha}m_1m_2$ and $n=n_1n_2$.

Let $p$ be a prime factor of $m_1$. Then considering equation \eqref{i} modulo $p$, we have:
$$ \left( n^2 \right )^Z+ \left( n^2 \right)^X \equiv 0 \pmod p \Rightarrow  n^{2(Z-X)} \equiv -1 \pmod p ,$$
Therefore, $n^{4 \frac{Z-X}{2}} \equiv -1 \pmod {p}$, so there exists an element of order 8 in $ \mathbb{Z}/ p \mathbb{Z}$ and $p \equiv 1 \pmod 8 $, hence $m_1 \equiv 1 \pmod 8 $. 

Next, let $p$ be a prime factor of $n_2$. Then tacking equation \eqref{ii} modulo $p$, we obtain
 $$ \left( m^2 \right )^Z+ \left( m^2 \right)^X \equiv 0 \pmod p \Rightarrow  m^{2(Z-X)} \equiv -1 \pmod p ,$$
by the same argument as above, we conclude that $n_2 \equiv 1 \pmod 8 $ and, therefore, $n_1 \equiv 5 \pmod 8 $.
Consider the equation \eqref{iii} modulo 8. Since $ \alpha \geq 2 $ we get:
$$ 1 \equiv \left (m_1n_1 \right) ^y \pmod 8, $$
therefore, $y$ is also even, and this completes the proof for this case.
  \subsection{$4 \pdiv m$ and $ n \equiv 7 \pmod 8 $} 
  In this case, from part 3 of lemma \ref{quadratic}, since $m +n \equiv 3 \pmod 5 $ $z$ is even, and from part 4 of lemma \ref{quadratic}, since $m-n \equiv 5 \pmod 8 $, $y$ is even. Hence, if $x,y,z$ is a solution for \eqref{bakhshakhar}, where $4 \pdiv m$ and $n \equiv 7 \pmod8 $, then $x \equiv y \equiv z \equiv 0 \pmod 2. $
To summarize, if $m \pdiv 4 $ and the pair $(n,m)$ satisfies the conditions of theorem \ref{main thereom} or \ref{main2}, then Je\'{s}manowicz' conjecture holds for the primitive Pythagorean triple $m^2-n^2,2mn,m^2+n^2.$


%
%


 
 \end{document}